\documentclass[a4paper,preprint,12pt]{elsarticle}
\usepackage{xcolor}
\definecolor{cblue}{RGB}{0,70,140}
\definecolor{cgreen}{RGB}{100,140,0}
\definecolor{cred}{RGB}{190,10,50}

\iffalse 
\newcommand{\therese}[1]{{\color{brown} TB says: #1}}
\newcommand{\dinis}[1]{{\color{blue} Dinis says: #1}}
\usepackage{todonotes}
\else
\newcommand{\therese}[1]{}
\newcommand{\dinis}[1]{}
\usepackage[disable]{todonotes}
\fi

\usepackage[hidelinks,colorlinks,pagebackref]{hyperref} 
\hypersetup{citecolor=cgreen,linkcolor=cblue,urlcolor=cblue}
\usepackage{amsmath,amssymb,amsfonts,amsthm,stmaryrd,mathtools} 
\usepackage{csquotes} 
\usepackage{cleveref}
\usepackage{cellspace} 
\usepackage{subcaption}

\newtheorem{theorem}{Theorem}

\newtheorem{conjecture}{Conjecture}
\newtheorem{definition}{Definition}

\newtheorem{question}{Question}

\newtheorem{corollary}{Corollary}
\newtheorem{observation}{Observation}
\crefname{observation}{observation}{observations}
\crefname{question}{question}{questions}

\newcommand{\etal}{\textit{et al.} }
\newcommand{\bipMinorOf}{\leq_{B}}
\newcommand{\minorOf}{\leq_{M}}
\newcommand{\curlybrackets}[1]{\left\{#1\right\}}
\newcommand{\IN}{\mathbb{N}}
\renewcommand{\mod}[1]{\text{ (mod }#1\text{)}}
\newcommand{\scale}{0.33}

\begin{document}

\begin{abstract}
    In \enquote{Bipartite minors} [Journal of Combinatorial Theory, Series B, 2016], Chudnovsky \etal introduced the \emph{bipartite minor relation} (a quasi-order on the class of bipartite graphs  similar
    to the \emph{minor relation} on general graphs) and asked whether it is a well-quasi-order.
    We answer this question negatively by giving an infinite set of $2$-connected bipartite graphs that are pairwise incomparable with respect to the bipartite minor relation.
    We additionally give two infinite sets 
    of pairs of bipartite graphs: one set of pairs $G,H$ such that $H$ is a bipartite minor, but not a minor, of $G$, and one set of pairs $G,H$ such that  $H$ is a minor, but not a bipartite minor, of $G$.
\end{abstract}

\begin{frontmatter}
    \title{Bipartite Graphs Are Not Well-Quasi-Ordered by Bipartite Minors}

    \author[1]{Therese Biedl\fnref{fn1}}
    \ead{biedl@uwaterloo.ca}
    
    \author[2]{Dinis Vitorino\corref{cor1}\fnref{fn2}}
    \ead{dinis@addition.pt}
    
    \affiliation[1]{organization={David R.~Cheriton School of Computer Science, University of Waterloo},
        addressline={200 University Avenue West},
        city={Waterloo},
        postcode={ON N2L 3G1},
        state={Ontario},
        country={Canada}
    }

    \affiliation[2]{organization={Department of Applied Mathematics, Faculty of Mathematics and Physics, Charles University},
        addressline={Malostranské náměstí 25},
        city={Prague},
        postcode={11800},
        country={Czech Republic}
    }
    
    \cortext[cor1]{Corresponding author}
    \fntext[fn1]{Research supported by NSERC, FRN RGPIN-2020-03958.}
    \fntext[fn2]{The majority of the research involved in this paper was conducted while the second author was a Master's student at the University of Waterloo.}
\end{frontmatter}

\section{Introduction}
\label{sec: intro}
Interesting classes of graphs often admit characterizations in terms of graphs that their elements do not contain with respect to a quasi-order.
Early results of this sort include K\H{o}nig's characterization of bipartite graphs in terms of forbidden subgraphs \cite{graphenAnwendungKonig}, Kuratowski's characterization of planar graphs in terms of forbidden topological minors \cite{Kuratowski1930}, and Wagner's extension of the latter to a set of forbidden minors \cite{wagnerDecompTheorem}.
Sometimes, one is interested only in the elements of a minor-closed class that have some additional properties.
For instance, one could be interested only in the bipartite graphs in some minor-closed class.
In \cite{bipartiteMinors}, Chudnovsky \etal introduced bipartite minors which give rise to a quasi-order that is somewhat similar to that generated by the \enquote{standard} minor relation.
Bipartite minors of bipartite graphs, however, have the advantage of always being bipartite \cite[p.~221]{bipartiteMinors}.

It is a famous result \cite{RSGraphMinorsXX} that the \enquote{standard} minor relation is a well-quasi-order, a term we define below.
Problem 2.7 of \cite{bipartiteMinors} asks whether the same holds for the bipartite minor relation, i.e., whether it is a well-quasi-order on the class of bipartite graphs.
In this short note, we answer this question negatively (see \Cref{subsec: bip minors and well-order}).
In addition, we consider interactions between the minor and bipartite minor relations:
There are bipartite graphs with bipartite minors that are not minors of theirs (see \Cref{subsec: bipartite minor but not minor}), and there are bipartite graphs with minors that are bipartite but not bipartite minors of theirs (see \Cref{subsec: minor but not bip minor}).
The latter constructions hold the key to disproving that bipartite graphs are well-quasi-ordered by bipartite minors.

It should be pointed out that other variants of the minor relation that are compatible with bipartite-ness were studied in \cite{RecCharMaxBipPlanarGraphs}.
One of the relations studied in \cite{RecCharMaxBipPlanarGraphs}, which the authors refer to as the bipartite-minor relation, is a well-quasi-order on the class of bipartite graphs.
This relation is, however, different from the one proposed in \cite{bipartiteMinors} and studied throughout this note here.

\section{Preliminaries}
\label{sec: prelims}
Let $G$ be a graph and $U$ be a vertex set.
The graph $G-U$ is the graph obtained from $G$ by deleting each vertex in $U$ (together with all edges incident to each of them).
Similarly, for an edge set $F$, $G-F$ is the graph obtained from $G$ by deleting each edge in $F$.
A graph that can be obtained from $G$ via vertex- and edge-deletions is a \emph{subgraph} of $G$.

A \emph{path} $P$ in $G$ is a sequence of vertices $V(P)=\curlybrackets{v_0,\dots,v_{k-1}}$ such that for all $i\in[0,k-2]$ the vertices $v_i$ and $v_{i+1}$ are adjacent.
The vertices $v_0$ and $v_{k-1}$ are the \emph{endpoints} of $P$; the other vertices are its \emph{internal vertices}.
A path $P$ in $G$ is \emph{induced} if no two non-consecutive vertices of $P$ are adjacent.
We will denote the graph isomorphic to the path on $k$ vertices by $P_k$ (and generally treat isomorphic graphs as if they were the same object).
A graph $G$ is \emph{connected} if for any pair of vertices in $G$, there is a path in $G$ containing both.
Otherwise it is \emph{disconnected}.
$G$ is \emph{$k$-connected} if for all vertex sets $U$ with cardinality at most $k-1$, the subgraph $G-U$ is connected.

\therese{Sorry, but I really want to keep the order so that we don't use things before they are defined.   This means defining connectivity before cycles, and moving Menger's down.}
\dinis{Very fair (I think I added the mention to Menger's theorem after changing the order and missed that the order had become weird)}

A \emph{cycle} $C$ in $G$ is a sequence of vertices $V(C)=\curlybrackets{v_0,\dots,v_{k-1}}$, $k\geq 3$, such that for all $i\in[0,k-1]$ the vertices $v_i$ and $v_{(i+1)\mod k}$ are adjacent.
A \emph{chord} of $C$ is an edge $v_iv_j$ between non-consecutive vertices in $C$, i.e., with $j-i\mod k\notin\{1,k-1\}$.
A cycle $C$ in $G$ is \emph{induced} if it has no chord.
A cycle $C$ is \emph{non-separating} if the number of connected components of $G$ is no smaller than that of $G-V(C)$.
We will denote the cycle on $k$ vertices by $C_k$.
A \emph{forest} is a graph with no cycles.
A \emph{tree} is a connected forest.
Menger's theorem implies that a graph $G$ is $2$-connected iff any two vertices of $G$ lie in a common cycle.

The \emph{contraction} of a vertex set $U$ in a graph $G$ is the graph $G/U$ obtained from $G$ by removing all vertices in $U$, adding a new vertex $v_U$, and making $v_U$ adjacent to all neighbours of vertices in $U$ that are not themselves in $U$.
An \emph{edge-contraction} is a contraction $G/\{u,v\}$ for an edge $uv\in E(G)$.
Edge-contractions, together with edge-, and vertex-deletions give rise to the minor relation in graphs.
Though standard, we define it below for comparison with the less standard bipartite minor relation.
\begin{definition}
    Let $G$ and $H$ be graphs.
    The graph $H$ is a \emph{minor} of $G$, denoted by $H\minorOf G$, if there is a sequence of graphs $H=H_0,H_1,\dots,H_k=G$ such that for all $i\in[0,k-1]$, the graph $H_i$ can be obtained from $H_{i+1}$ by deleting a vertex, deleting an edge, or contracting an edge.
\end{definition}
The definition of the bipartite minor relation is almost identical.
However, it requires the introduction of a new operation on graphs, namely an \emph{admissible contraction}.
\begin{definition}[{\cite[p. 221]{bipartiteMinors}}]
    Let $G$ be a graph and $u,v$ be vertices in $G$.
    The contraction of the vertices $u$ and $v$ is \emph{admissible} if $u$ and $v$ have a common neighbour $w$ such that the path $u,w,v$ is part of an induced non-separating cycle in $G$.
\end{definition}

\begin{definition}[{\cite[p. 221]{bipartiteMinors}}]
    Let $G$ and $H$ be graphs.
    The graph $H$ is a \emph{bipartite minor} of $G$, denoted by $H\bipMinorOf G$, if there is a sequence of graphs $H=H_0,H_1,\dots,H_k=G$ such that for all $i\in[0,k-1]$, the graph $H_i$ can be obtained from $H_{i+1}$ by deleting a vertex, deleting an edge, or performing an admissible contraction.
\end{definition}

A \emph{quasi-order} is a binary relation that is \emph{transitive} and \emph{reflexive}.
The subgraph, minor, and bipartite minor relations are quasi-orders on the class of all graphs.
A \emph{well-quasi-order} is a quasi-order $\leq$ on a class $\mathcal{K}$ such that every infinite sequence $x_{1},x_2,\dots$ of elements of $\mathcal{K}$ contains a pair $x_i\leq x_j$ with $i<j$.
The subgraph relation is known not to be a well-quasi-order on the class of all graphs, not even of all trees.
To see this, let $F_l$ (for $l\geq 2$) be the tree obtained by taking a path with $l$ vertices and attaching two vertices of degree 1 at each of its endpoints, \therese{the definition of $F_l$ seemed unnecessarily long and I shortened it}
\dinis{This is much better, thank you!}
see \Cref{fig:inf antichain forest subgraph}.
It is known that $\{F_l\}_{l\geq 2}$ is an infinite set of trees that are pairwise incomparable with respect to the subgraph relation \cite{forks}.

\begin{figure}[ht]
    \hspace*{\fill}
    \includegraphics[scale=0.7,page=1]{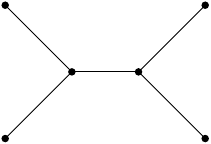}
    \hspace*{\fill}
    \includegraphics[scale=0.7,page=2]{forest_1.pdf}
    \hspace*{\fill}
    \includegraphics[scale=0.7,page=3]{forest_1.pdf}
    \hspace*{\fill}
    \caption{The three smallest elements of an infinite set of forests that are pairwise incomparable with respect to the subgraph relation.  }
    \label{fig:inf antichain forest subgraph}
\end{figure}

\section{Bipartite Minors vs. Minors}
\label{sec: bip minors vs minors}
In their paper \cite{bipartiteMinors} Chudnovsky et al.~studied multiple minor-closed graph classes (specifically, planar graphs, outer-planar graphs, and forests), and characterized all bipartite graphs in each of these in terms of a set of forbidden bipartite minors.
Naturally one wonders whether the bipartite graphs in \emph{every} minor-closed graph class $\mathcal{K}$ can be similarly characterized via a set of forbidden bipartite minors.
For such a characterization to exist, the set of bipartite graphs in class $\mathcal{K}$ would have to be closed under bipartite minors.
If $\mathcal{K}$ consists of a bipartite graph $G$ and all its minors, then this would require every bipartite minor of $G$ to also be in $\mathcal{K}$, i.e., a minor of $G$.
Hence, we first need to answer the following question.
\begin{question}
    \label{question: bip minor implies minor}
    Suppose $G$ and $H$ are bipartite graphs with $H$ a bipartite minor of $G$.
    Is $H$ necessarily a minor of $G$?
\end{question}
Despite not having been explicitly asked, the authors of \cite{bipartiteMinors} were clearly aware of \Cref{question: bip minor implies minor} and in fact answered it implicitly in two different ways:
On page 223 of \cite{bipartiteMinors}, they say that there are graphs (without explicitly giving any) where admissible contractions preserve neither linkless-embeddability nor embeddability into closed surfaces other than the plane.  \todo{new!}So if (say) $G$ is a linkless-embeddable graph for which the bipartite minor $H$ is \emph{not} linkless- embeddable, then $H\bipMinorOf G$ but $H\not\minorOf G$ since any minor of $G$ is linkless-embedable.
%
\dinis{Do we not need the equivalence here?}\therese{I don't think so, see what I added above.   But as always I am easily confused in the `nots' here, so check my argument.}\dinis{Ok, I don't know why I thought we needed equivalence. This seems to work}
%
Additionally, Remark 2.3 of \cite{bipartiteMinors} states that the (bipartite) graph obtained from $K_5$ by subdividing each edge precisely once has a $K_{3,3}$ bipartite minor, though it has no subgraph homeomorphic to $K_{3,3}$ (and hence has no $K_{3,3}$-minor).
Since both arguments are somewhat indirect, we feel that it is worth giving an explicit example (with an easy direct argument).
In fact, we give infinitely many pairs of bipartite graphs $G,H$ such that $H$ is a bipartite minor, but not a minor, of $G$  (see \Cref{subsec: minor but not bip minor}).

Given \Cref{question: bip minor implies minor}, it is natural to think about its converse: If $H$ is a minor of $G$ (and both graphs are bipartite), is $H$ a bipartite minor of $G$?
Unfortunately, this too is false.
Note that forests has no cycles, and so admit no admissible contractions.
Hence, a forest $F_1$ is a bipartite minor of a forest $F_2$ if and only if $F_1$ is a subgraph of $F_2$.
However, the forests in \Cref{fig:inf antichain forest subgraph} are pairwise incomparable with respect to the subgraph relation, but are from left to right increasingly ordered by the minor relation.
This is a somewhat unsatisfactory answer, since it relies on a complete lack of cycles.
What happens if we require any two vertices in $G$ (or $H$) to be in a common cycle, i.e., that $G$ (or $H$) is $2$-connected?
\begin{question}
    \label{question: minor implies bip minor}
    Suppose $G$ and $H$ are $2$-connected bipartite graphs with $H$ a minor of $G$.
    Is $H$ necessarily a bipartite minor of $G$?
\end{question}
Once more, we construct infinitely many pairs of graphs proving that the answer to this question must be negative (cf.~\Cref{subsec: minor but not bip minor}).

Problem 2.7 of \cite{bipartiteMinors} asks whether the bipartite  minor relation is a well-quasi-order on the class of bipartite graphs.
By the same argument as above, this is obviously false for any class that contains infinitely many forests as those in \Cref{fig:inf antichain forest subgraph}.
However, one could once again wonder about what happens to the class of $2$-connected bipartite graphs.
\begin{question}
    \label{question: wqo of bipartite minors}
    Is the bipartite minor relation a well-quasi-order on the class of $2$-connected bipartite graphs?
\end{question}
The insights we use to answer \Cref{question: minor implies bip minor} immediately imply a negative answer to \Cref{question: wqo of bipartite minors} and hence solve a stronger form of Problem 2.7 of \cite{bipartiteMinors} (cf.~\Cref{subsec: bip minors and well-order}).

\subsection{Bipartite Minor, but Not Minor}
\label{subsec: bipartite minor but not minor}
We require the following graph class (referred to by this name in \cite{tadpoleGraphIntro} and studied under a different name as early as \cite{tadpoleGraphIntro2}):
\dinis{I don't know whether we want to mention that other people had studied these graphs before?
If not, delete the parentheses.}
The \emph{tadpole} $T$ with a \emph{head} of length $l$ and a \emph{pole} of length $l_1$ is denoted by $T(l,l_1)$, and is the graph obtained from a cycle $C_l$ and a path $P_{l_1}$ by picking a vertex in $C_l$ and adding an edge between it and an endpoint of $P_{l_1}$.
See \Cref{fig:tadpoles}.
\begin{figure}[ht]
    \centering
    \begin{subfigure}[t]{0.3\linewidth}
        \centering
        \includegraphics[angle=-90,scale=0.3,page=1]{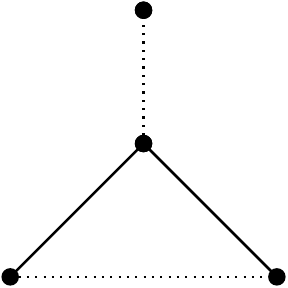}
        \caption{A tadpole.}
    \end{subfigure}
    \hfill
    \begin{subfigure}[t]{0.3\linewidth}
        \centering
        \includegraphics[angle=-90, scale=0.3,page=2]{tadpoles.pdf}
        \caption{$T(3,1)$.}
    \end{subfigure}
    \hfill
    \begin{subfigure}[t]{0.3\linewidth}
        \centering
        \includegraphics[angle=-90,scale=0.3,page=3]{tadpoles.pdf}
        \caption{$T(4,3)$.}
    \end{subfigure}
    \caption{Some examples of tadpoles. Here and elsewhere, dashed lines indicate edges or induced paths.}
    \label{fig:tadpoles}
\end{figure}

\begin{theorem}
    \label{thm: bip minor but not minor}
    For all integers $l\geq 3$ and $l_1\geq 1$, the tadpole $T(l,l_1)$ is a bipartite minor of $C_{p}$ for any $p\geq l+2l_1$ with the same parity as $l$, but $T(l,l_1)$ is not a minor of $C_p$ for any integer $p\geq 3$.
\end{theorem}

\begin{proof}
    Consider the graph $C_p$ where $p=l+2k$ and $k\geq l_1$.
    Contract a pair of vertices of $C_p$ that have a common neighbour;
    this yields $T(p-2,1)$.
    Let $w$ be the vertex of the head of $T(p-2,1)$ with a neighbour on its pole, and let $u,v$ be its two neighbours on the head.
    The contraction of $u,v$ is admissible and yields $T(p-4,2)$.
    \therese{switched $v$ and $w$ here to match how we defined contractions earlier}
    Repeating $k-2$ times (always contracting the two vertices of the head that have the vertex adjacent to the pole as a common neighbour) gives $T(p-2k,k)=T(l,k)$.
    Deleting the $k-l_1$ vertices at the end of the pole that is furthest from the head now gives $T(l,l_1)$ as desired.

    To see that $T(l,l_1)\not\minorOf C_{p}$ for any $p\in\IN_{\geq 3}$, one need only observe that $C_p$ has maximum degree two, but $T(l,l_1)$ has a vertex of degree three.
    As one can easily check, the graphs with maximum degree at most two form a minor-closed class (although taking minors does not generally preserve maximum degree), yielding the desired claim.
    \dinis{I reworded a bit here (see discussion/commented out things above)}
    \therese{I don't see any rewording except the `subquadratic', but I think what's here is fine.}
    \therese{I had not heard the term `subquadratic graph' before, and would have guessed that this means `has asymptotically fewer than $n^2$ edges'.  Let's not use this term.  (Or if you really love it, define it.)}
    \dinis{Fair.
    I'm not sure the term subquadratic exists anywhere else.
    I mostly wanted to avoid using the words maximum degree and minor-closed to prevent people from thinking that any upper bound on max degree is preserved by taking minors (which is obviously false)}
\end{proof}
In particular, there exist infinitely many pairs of bipartite graphs $G$ and $H$ with $H\bipMinorOf G$, but $H\not\minorOf G$.
Also note that the graph class $\mathcal{K}$ consisting of all graphs with maximum degree at most 2 is minor-closed, but its bipartite members are not closed under taking bipartite minors.
\dinis{This is true, but I'm not sure it's worth saying (I don't really see why this is that interesting)}
\therese{We motivated Q1 with `is there a minor-closed graph class that is not closed under taking bipartite minors', and here is the answer to that.}
\dinis{Yeah, but that was sort of known (for any surface, the graphs that embed on that surface form a minor-closed class, and Chudnovsky\etal~claimed that bipartite minors don't preserve embeddability on any closed surface other than the plane), but sure, it hadn't been carefully proved before, and now it is.
I'm ok with keeping it as well.}

\subsection{Minor, but Not Bipartite Minor}
\label{subsec: minor but not bip minor}
We introduce another graph class:
The \emph{dog} with a \emph{snout} of length $l\geq 3$ and $k\leq l/2$ \emph{ears} of length $l_1,\dots,l_k$ all greater than two is denoted by $D(l,l_1,\dots,l_k)$, and is the graph obtained from a cycle $C_l$ and cycles $C_{l_1},\dots,C_{l_k}$ by picking $2k$ vertices $v_1,\dots,v_{2k}$ that are consecutive in $C_l$ and two consecutive vertices $w_{i},x_{i}$ in each $C_{l_i}$ for $i\in[1,k]$, and identifying $v_{2i-1},v_{2i}$ and $w_{i},x_{i}$ (as well as the edges $v_{2i-1}v_{2i}$ and $w_{i}x_{i}$) for each $i\in[1,k]$.
See \Cref{fig:dog example}.
\begin{figure}[ht]
    \centering
    \begin{subfigure}[t]{0.3\linewidth}
        \centering
        \includegraphics[width=0.7\linewidth,page=1]{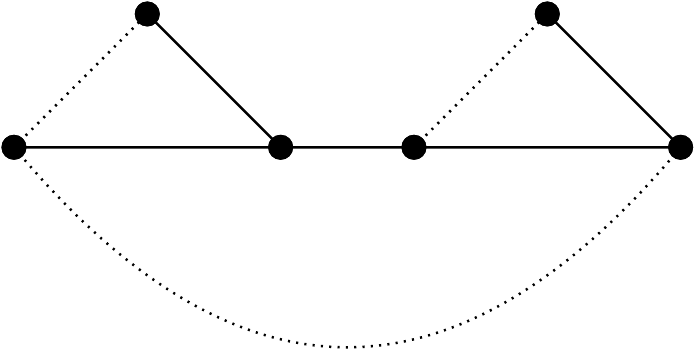}
        \caption{A two-eared dog.}
    \end{subfigure}
    \hfill
    \begin{subfigure}[t]{0.3\linewidth}
        \centering
        \includegraphics[width=0.7\linewidth,page=3]{dogs.pdf}
        \caption{$D(10,4,4)$.}
    \end{subfigure}
    \hfill
    \begin{subfigure}[t]{0.3\linewidth}
        \centering
        \includegraphics[width=0.7\linewidth,page=2]{dogs.pdf}
        \caption{$D(6,3,6,5)$.}
    \end{subfigure}
    \caption{Some examples of dogs.}
    \label{fig:dog example}
\end{figure}

\begin{theorem}
    \label{thm: minor but not bip minor}
    For all integers $l\geq 5$, $l_1,l_2\geq 3$, and $k\geq 1$, the dog $D(l,l_1,l_2)$ is a minor, but not a bipartite minor, of the dog $D(l+k,l_1,l_2)$.
\end{theorem}

Before proving \Cref{thm: minor but not bip minor}, note that admissible contractions only permit one to contract vertices $v,w$ on a common cycle.
In particular $v,w$ must lie in a common \emph{block} of $G$, i.e., a maximal $2$-connected subgraph.
This gives the following.
\begin{observation}
    \label{obs: 2-conn bip minor implies bip minor of block}
    If $H$ is a $2$-connected bipartite minor of $G$, then $H$ is a bipartite minor of a block of $G$.
\end{observation}

\begin{corollary} 
    \label{cor: 2-conn bip minor implies bip minor of block}
    All $2$-connected bipartite minors of a cycle are themselves cycles or 
    \emph{trivial blocks} (copies of $P_2$).
    All $2$-connected bipartite minors of a one-eared dog are one-eared dogs, cycles, 
    or trivial blocks.
\end{corollary}

\begin{proof}[Proof of \Cref{thm: minor but not bip minor}]
    Contracting two adjacent vertices that are on the snout and at least one of which is not in either ear of $D(l+k,l_1,l_2)$ yields $D(l+k-1,l_1,l_2)$.
    Thus, proceeding by induction over $k$, we have that $D(l,l_1,l_2)\leq_{M}D(l+k,l_1,l_2)$.

    To see that $D(l,l_1,l_2)\not\leq_{B}D(l+k,l_1,l_2)$, consider the blocks of a graph $H'$ obtained from $D(l+k,l_1,l_2)$ using one of the operations permitted for bipartite minors.   We claim that the blocks of $H'$ are trivial blocks, 
    one-eared dogs, cycles, or two-eared dogs with at least one strictly shorter ear.
    Clearly this holds if $H'$ is obtained by deleting a vertex or an edge.
    Now assume that $H'$ is obtained via an admissible contraction, say of vertices $u,v$ with common neighbour $w$ such that $u,w,v$ is part of an induced non-separating cycle of $D(l+k,l_1,l_2)$.
    Note that the snout of $D(l+k,l_1,l_2)$ is separating, and so $u,w,v$ is not fully on the snout.
    Further, the only cycle containing vertices from both ears is the unique Hamiltonian cycle of $D(l+k,l_1,l_2)$ which is not induced.
    Thus, some of $u,w,v$ must be part of an ear, say of length $l_i$.
    We distinguish cases based on whether they also belong to the snout.
    Up to renaming, we assume that if one of $u,v$ is on the snout, then this holds for $u$.
    \Cref{fig:case analysis proof minor but not bip minor} illustrates these cases.
        \begin{itemize}
            \item [1.] If neither $u$ nor $w$ is on the snout, then (by assumption) neither is $v$, which implies $l_i\geq 5$.
            The blocks of $H'$ are a two-eared dog where the sum of the lengths of the ears is $l_1+l_2-2$, and a trivial block.
            \item [2.] If $u$ is on the snout and $w$ is not, then we consider two sub-cases.
            If $l_i\geq 5$, then the contraction reduces the ear to length $l_i-2$ (and creates a trivial block).
            If $l_i\leq 4$, then the contraction removes the ear and adds a trivial block; for $l_i=3$ it also shortens the snout.
            \item [3.] If both $u$ and $w$ are on the snout, then we again have sub-cases.
            If $l_i\geq 5$, then the contraction changes the ear to a cycle of length $l_i-2$ attached to the snout at the contracted vertex. 
            If $l_i\leq 4$, then the contraction removes the ear; for $l_i=4$ it also adds a trivial block. 
        \end{itemize}
    
    \begin{figure}[ht!]
        \centering
        \begin{subfigure}[t]{0.47\linewidth}
            \centering
            \includegraphics[scale=\scale]{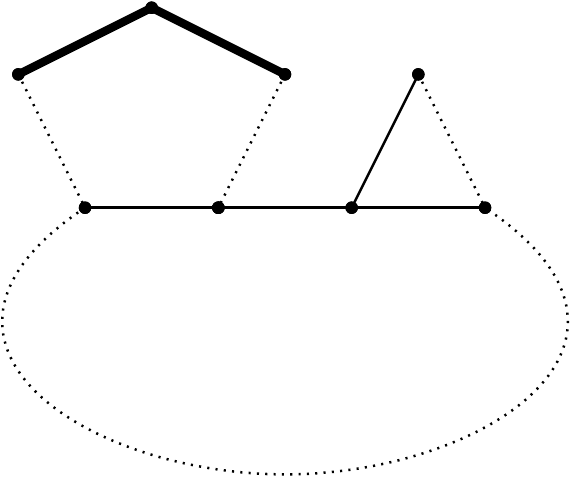}
            \caption{Contraction when neither of $u,w$ is on the snout.}
            \label{subfig: two-eared dog ear len geq 5}
        \end{subfigure}
        \hfill
        \begin{subfigure}[t]{0.47\linewidth}
            \centering
            \includegraphics[scale=\scale]{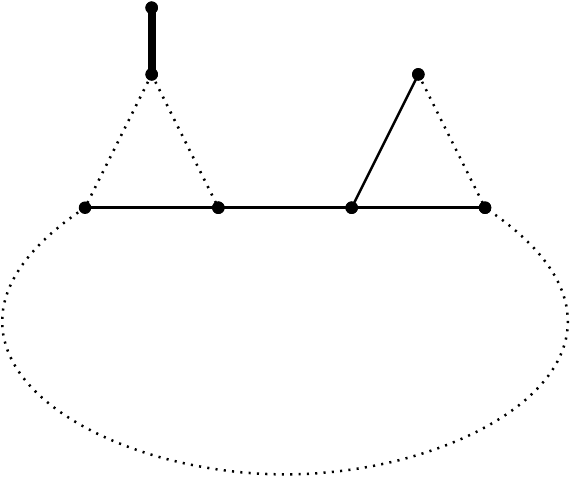}
            \caption{The resulting graph.}
        \end{subfigure}
    
        \vspace*{10pt}
    
        \begin{subfigure}[t]{0.47\linewidth}
            \centering
            \includegraphics[scale=\scale]{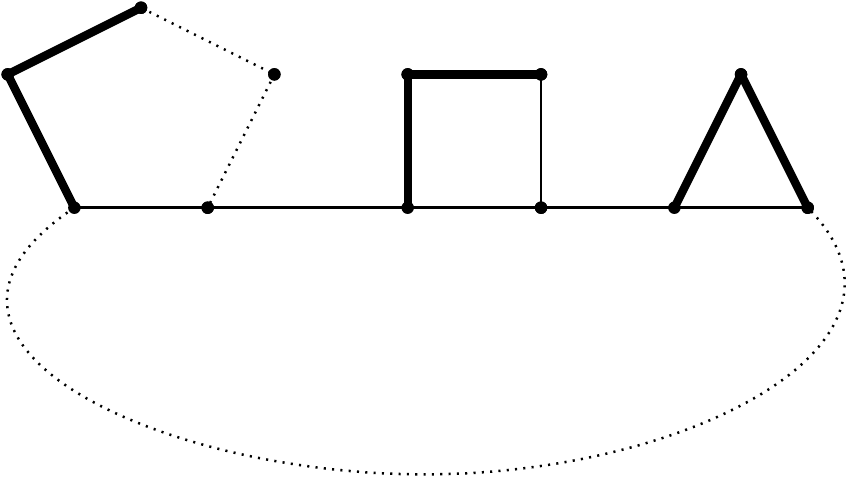}
            \caption{Contraction when $u$ is on the snout and $w$ is not.}
        \end{subfigure}
        \hfill
        \begin{subfigure}[t]{0.47\linewidth}
            \centering
            \includegraphics[scale=\scale]{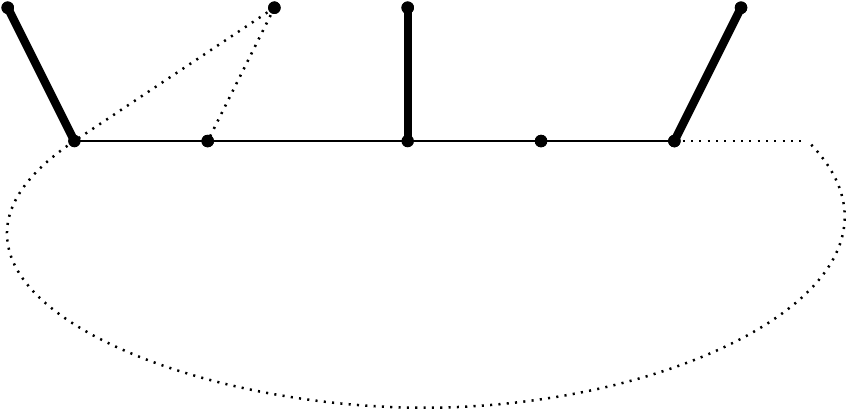}
            \caption{The resulting graph.}
        \end{subfigure}
    
        \vspace*{10pt}
    
        \begin{subfigure}[t]{0.47\linewidth}
            \centering
            \includegraphics[scale=\scale]{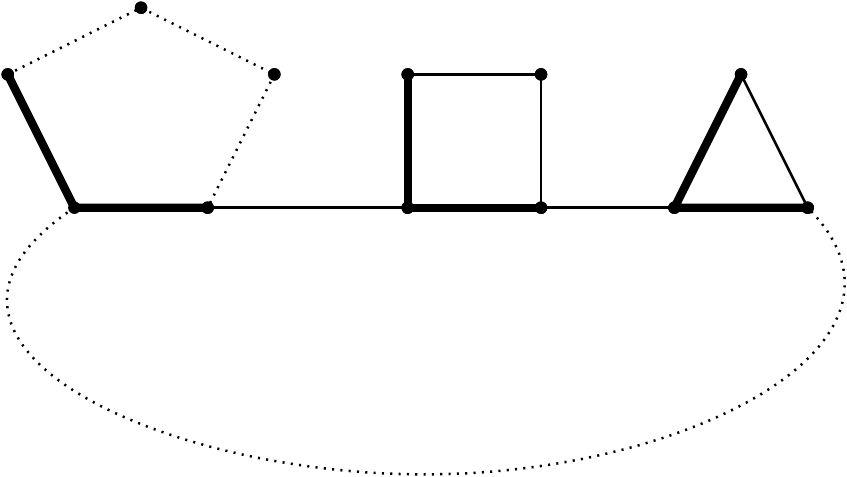}
            \caption{Contraction when $u,w$ are both on the snout.}
        \end{subfigure}
        \hfill
        \begin{subfigure}[t]{0.47\linewidth}
            \centering
            \includegraphics[scale=\scale]{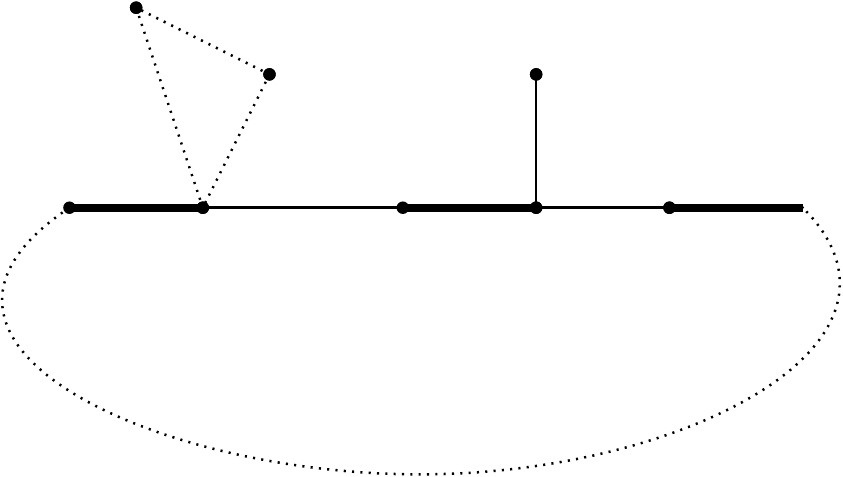}
            \caption{The resulting graph.}
        \end{subfigure}

        \caption{The admissible contractions contemplated in the proof of \Cref{thm: minor but not bip minor}.
        Thick edges illustrate the path $u,w,v$ (we show multiple cases of $l_i$ by using multiple ears).}
        \label{fig:case analysis proof minor but not bip minor}
    \end{figure}
    So in all cases the blocks of $H'$ are as desired  (trivial, cycles, one-eared dogs or two-eared dogs with at least one strictly shorter ear).
    Combining an inductive argument over the sum of the lengths of the ears with \Cref{cor: 2-conn bip minor implies bip minor of block}, we have that $D(l,l_1,l_2)$ is not a bipartite minor of $D(l+k,l_1,l_2)$.
\end{proof}

\subsection{Bipartite Minors and Well-Quasi-Order}
\label{subsec: bip minors and well-order}
We finish with a negative answer to \Cref{question: wqo of bipartite minors}.
\begin{theorem}
    \label{thm: bip minor not a well ordering}
    The bipartite minor relation is not a well-quasi-order, not even on the class of $2$-connected bipartite graphs.
\end{theorem}
\begin{proof}
    It follows immediately from \Cref{thm: minor but not bip minor} that the graphs in the infinite set
    $\curlybrackets{D(k,4,4)\mid k\text{ a positive even number at least four}}$ are pairwise incomparable with respect to the bipartite minor relation.
    Furthermore, they are bipartite and $2$-connected.
    Hence, the bipartite minor relation is not a well-quasi-order on the class of $2$-connected bipartite graphs.
\end{proof}

Despite being of some importance in our proof of \Cref{thm: minor but not bip minor}, we conjecture that the number two plays no special role in this respect.
\begin{conjecture}
    There exists no $k$ such that the class of $k$-connected bipartite graphs is well-quasi-ordered by the bipartite minor relation.
\end{conjecture}

\section{Acknowledgments}
We would like to thank Christopher Burcell and the reviewers of an earlier version of this paper for helpful input.

\bibliographystyle{elsarticle-num.bst}
\bibliography{bibliography}
\end{document}